\documentclass[12pt]{article}
\usepackage{amsmath,amssymb,amsthm}

\newtheorem{theorem}{Theorem}
\newtheorem{lemma}[theorem]{Lemma}
\newtheorem{corollary}[theorem]{Corollary}

\def\barr{\begin{array}}
\def\earr{\end{array}}

\title{A note on the Chermak-Delgado lattice of a finite group}
\author{Marius T\u arn\u auceanu}
\date{December 9, 2016}

\begin{document}

\maketitle

\begin{abstract}
    In this note we describe the structure of finite groups $G$ whose Chermak-Delgado lattice is the interval $[G/Z(G)]=\{H\in L(G) \mid Z(G)\leq H\leq G\}$.
\end{abstract}

{\small
\noindent
{\bf MSC2000\,:} Primary 20D30; Secondary 20D15, 20E34.

\noindent
{\bf Key words\,:} Chermak-Delgado measure, Chermak-Delgado lattice, centralizer lattice, subgroup lattice.}

\section{Introduction}

Let $G$ be a finite group and $L(G)$ be the subgroup lattice of $G$. The \textit{Chermak-Delgado measure} of a subgroup $H$ of $G$ is defined by
\begin{equation}
m_G(H)=|H||C_G(H)|.\nonumber
\end{equation}Let
\begin{equation}
m(G)={\rm max}\{m_G(H)\mid H\leq G\} \mbox{ and } {\cal CD}(G)=\{H\leq G\mid m_G(H)=m(G)\}.\nonumber
\end{equation}Then the set ${\cal CD}(G)$ forms a modular self-dual sublattice of $L(G)$,
which is called the \textit{Chermak-Delgado lattice} of $G$. It was first introduced by Chermak and Delgado \cite{5}, and revisited by Isaacs \cite{7}. In the last years
there has been a growing interest in understanding this lattice, especially for $p$-groups (see e.g. \cite{1,2,3,10}). The study can be naturally extended to nilpotent groups,
since by \cite{2} the Chermak-Delgado lattice of a direct product of finite groups decomposes as the direct product of the Chermak-Delgado lattices of the factors. Recall also that
if $H\in {\cal CD}(G)$, then $C_G(H)\in {\cal CD}(G)$ and $C_G(C_G(H))=H$. This implies that ${\cal CD}(G)$ is contained in the centralizer lattice $\mathfrak{C}(G)$ of $G$.

We remark that ${\cal CD}(G)=[G/Z(G)]$ for many finite groups $G$, such us $D_8$, $Q_8$, any abelian group, ... and so on. Thus, the study of finite groups satisfying this property is very natural. It is the goal of the current note. Our main result is stated as follows.

\begin{theorem}\label{th:C1}
    Let $G$ be a finite group. Then ${\cal CD}(G)=[G/Z(G)]$ if and only if $G=G_1\times\cdots\times G_r\times A$, where
    \begin{equation}
    \gcd(|G_i|,|G_j|)=1=\gcd(|G_i|,|A|) \mbox{ for all } i\neq j,\nonumber
    \end{equation}$A$ is an abelian group, and every $G_i$ is a $p$-group satisfying
    \begin{equation}
    [G_i/Z(G_i)] \mbox{ is modular and } G'_i \mbox{ is cyclic}.
    \end{equation}
\end{theorem}

Note that the conditions (1) are equivalent with the conditions
\begin{equation}
    G'_i=\langle a\rangle \mbox{ is cyclic and } [\langle a\rangle,G_i]\leq\langle a^4\rangle
\end{equation}by Theorem 9.3.19 of \cite{9} (see also \cite{4,8}).
\bigskip

The following corollary is an immediate consequence of Theorem 1.

\begin{corollary}
    Every finite group $G$ satisfying ${\cal CD}(G)=[G/Z(G)]$ is nilpotent.
\end{corollary}

By Corollary 9.3.18 of \cite{9} (see also \cite{6}), we know that there are finite groups in which every subgroup is a centralizer. Theorem 1 shows that a similar result
does not hold for the Chermak-Delgado lattice.

\begin{corollary}
    There is no finite non-trivial group $G$ such that ${\cal CD}(G)=L(G)$.
\end{corollary}

Also, Theorem 1 shows that the property ${\cal CD}(G)=[G/Z(G)]$ is inherited by subgroups.

\begin{corollary}
    If\, $G$ is a finite group satisfying ${\cal CD}(G)=[G/Z(G)]$ and $H$ is a subgroup of $G$, then\, ${\cal CD}(H)=[H/Z(H)]$.
\end{corollary}

Observe that if for a finite group $G$ we have ${\cal CD}(G)=[G/Z(G)]$, then $\mathfrak{C}(G)=[G/Z(G)]$ and so $\mathfrak{C}(G)={\cal CD}(G)$ is a modular lattice.
Moreover, its length $l$ must be even by Lemma 9.3.10 of \cite{9}. Elementary examples of such groups are all abelian groups for $l=0$, and $D_8$, $Q_8$ for $l=2$. 
A more general example is the following.

\bigskip\noindent{\bf Example.} Let $G$ be an extra-special group $G$ of order $p^{2n+1}$. Then 
\begin{equation}
    G/Z(G)\cong\mathbb{Z}_p^{2n} \mbox{ is modular and } G'\cong\mathbb{Z}_p \mbox{ is cyclic},\nonumber
\end{equation}and therefore ${\cal CD}(G)=[G/Z(G)]$ is a modular lattice of length $2n$.
\bigskip

Finally, we indicate a natural open problem concerning the above study.

\bigskip\noindent{\bf Open problem.} Describe the structure of finite groups $G$ such that ${\cal CD}(G)$ is an interval (not necessarily $[G/Z(G)]$) of $L(G)$.

\section{Proof of the main result}

We start by proving an auxiliary result.

\begin{lemma}
    Let $G$ be a finite $p$-group. Then ${\cal CD}(G)=[G/Z(G)]$ if and only if $[G/Z(G)]$ is modular and $G'$ is cyclic.
\end{lemma}

\begin{proof}
   If ${\cal CD}(G)=[G/Z(G)]$, then $\mathfrak{C}(G)=[G/Z(G)]$ and so $[G/Z(G)]$ is modular and $G'$ is cyclic by Theorem 9.3.19 of \cite{9}.

   Conversely, if $[G/Z(G)]$ is modular and $G'$ is cyclic, then $\mathfrak{C}(G)=[G/Z(G)]$. By Lemma 4 of \cite{4} we infer that
   \begin{equation}
    m_G(H)=|H||C_G(H)|=|G||Z(G)|, \forall\, H\in [G/Z(G)],\nonumber
   \end{equation}that is all subgroups in $[G/Z(G)]$ have the same Chermak-Delgado measure. This shows that ${\cal CD}(G)=[G/Z(G)]$, as desired.
\end{proof}

\noindent{\bf Remark.} Let $G$ be a non-abelian $p$-group of order $p^n$ satisfying ${\cal CD}(G)=[G/Z(G)]$. If $G$ contains an abelian subgroup 
$M$ of order $p^{n-1}$ (as it happens for $D_8$ and $Q_8$), then $(G:Z(G))=p^2$. 

Indeed, we have $M\subseteq C_G(M)$ because $M$ is abelian and thus
\begin{equation}
    |G||Z(G)|=m(G)=|M||C_G(M)|\geq |M|^2.\nonumber
\end{equation}One obtains
\begin{equation}
    p^2\leq(G:Z(G))\leq(G:M)^2=p^2,\nonumber
\end{equation}that is $(G:Z(G))=p^2$. 
\bigskip

We are now able to prove our main theorem.

\begin{proof}[Proof of Theorem \ref{th:C1}]
    Assume first that $G=G_1\times\cdots\times G_r\times A$, where $G_i$, $i=1,...,r$, and $A$ satisfy the conditions in Theorem 1. Then
    \begin{equation}
    {\cal CD}(G_i)=[G_i/Z(G_i)], \forall\, i=1,...,r,\nonumber
    \end{equation}by Lemma 5. It follows that
    \begin{align*}
    {\cal CD}(G)
    &={\cal CD}(G_1)\times\cdots\times{\cal CD}(G_r)\times\{A\}\\
    &=[G_1/Z(G_1)]\times\cdots\times[G_r/Z(G_r)]\times\{A\}\\
    &=[G/Z(G)].\nonumber
    \end{align*}

    Conversely, assume that ${\cal CD}(G)=[G/Z(G)]$. Since ${\cal CD}(G)\subseteq \mathfrak{C}(G)$, we infer that $\mathfrak{C}(G)=[G/Z(G)]$. Then Theorem 9.3.17 of \cite{9} implies that
    $G=G_1\times\cdots\times G_r\times A$, where
    \begin{equation}
    \gcd(|G_i|,|G_j|)=1=\gcd(|G_i|,|A|) \mbox{ for all } i\neq j,\nonumber
    \end{equation}$A$ is an abelian group, and every $G_i$ is either a $\{p,q\}$-group with $|G_i/Z(G_i)|=pq$ or a $p$-group satisfying
    $\mathfrak{C}(G_i)=[G_i/Z(G_i)]$, $p$ and $q$ primes. Clearly, this leads to
    \begin{equation}
    {\cal CD}(G)={\cal CD}(G_1)\times\cdots\times{\cal CD}(G_r)\times\{A\}\nonumber
    \end{equation}and
    \begin{equation}
    [G/Z(G)]=[G_1/Z(G_1)]\times\cdots\times[G_r/Z(G_r)],\nonumber
    \end{equation}implying that 
    \begin{equation}
    {\cal CD}(G_i)=[G_i/Z(G_i)], \forall\, i=1,...,r.\nonumber
    \end{equation}If $G_i$ would be a $\{p,q\}$-group with $|G_i/Z(G_i)|=pq$ and $p<q$, then ${\cal CD}(G_i)$ would consists only of the unique subgroup of index $p$ contained in $[G_i/Z(G_i)]$, a contradiction. Consequently, $G_i$ is a $p$-group. On the other hand, it satisfies the conditions (1) by Theorem 9.3.19 of \cite{9}. This completes the proof.
\end{proof}

\vspace*{5ex}\small

\hfill
\begin{minipage}[t]{5cm}
Marius T\u arn\u auceanu \\
Faculty of  Mathematics \\
``Al.I. Cuza'' University \\
Ia\c si, Romania \\
e-mail: {\tt tarnauc@uaic.ro}
\end{minipage}

\end{document}